\documentclass[a4paper, reqno]{amsart}
\usepackage{pigpen}
\usepackage{stmaryrd}
\usepackage{tikz-cd}
\usepackage{mathrsfs}
\usepackage{array}
\usepackage{amsmath, amsthm, amscd, amssymb, latexsym, eucal}
\usepackage[all]{xy}
 \calclayout \makeatletter
 \calclayout \makeatletter
\def\serieslogo@{} \def\@setcopyright{} \makeatother
\makeatletter
\renewcommand*\env@matrix[1][c]{\hskip -\arraycolsep
  \let\@ifnextchar\new@ifnextchar
  \array{*\c@MaxMatrixCols #1}}
\makeatother

\newtheorem{thm}{Theorem}[section]
\newtheorem{cor}[equation]{Corollary}
\newtheorem{lem}[equation]{Lemma}
\newtheorem{prop}[equation]{Proposition}

\theoremstyle{definition}

\newtheorem{defn}[equation]{Definition}
\newtheorem{rem}[equation]{Remark}
\newtheorem{exam}[equation]{Example}

\newcommand{\C}{\mathscr C}

\newcommand{\K}{\mathcal K}

\newcommand{\T}{\mathcal T}

\newcommand{\W}{\mathcal W}

\newcommand{\ie}{i.e.\ }

\def\a{\alpha}

\def\e{\varepsilon}

\usepackage{tikz}

\newcommand{\po}{\ar@{}[dr]|{\text{\pigpenfont R}}}
\newcommand{\pb}{\ar@{}[dr]|{\text{\pigpenfont J}}}

\newcommand{\PDer}{\mathbf{PDer}}

\newsavebox{\proofbox}
\savebox{\proofbox}{\begin{picture}(7,7)%
  \put(0,0){\framebox(7,7){}}\end{picture}}

\newcommand{\Mod}{\text{-}\mathbf{Mod}}
\DeclareMathOperator{\dia}{dia}

\makeatletter
\newcommand{\xLeftarrow}[2][]{\ext@arrow 0359\Leftarrowfill@{#1}{#2}}
\makeatother

\makeatletter
\newcommand{\xRightarrow}[2][]{\ext@arrow 0359\Rightarrowfill@{#1}{#2}}
\makeatother

\usepackage{cite}
\numberwithin{equation}{section}

\usepackage{hyperref}

\begin{document}

\title[Levelwise modules over separable monads on stable derivators]{Levelwise modules over separable monads on stable derivators}
\author[I. Lagkas-Nikolos]{Ioannis Lagkas-Nikolos}

\begin{abstract}
We show that given a separable cocontinuous monad on a stable derivator, the levelwise Eilenberg-Moore categories of modules glue together to a stable derivator. As an application, we give examples of derivators that satisfy all the axioms for stability except the strongness one.
\end{abstract}

\maketitle

\setcounter{tocdepth}{1} \tableofcontents

\section{Introduction}
Let $\C$ be a tensor triangulated category. Given a monoid object $A$ in $\C$, we can consider the category of modules over it. This category, however, is usually not useful from a homotopic perspective. In particular, it might fail to inherit a natural triangulated structure from that of $\C$ (see Example~\ref{exam: modules not compatible with triangulation}). A notable exception is when the monoid $A$ in question is separable (see~\cite{Balmer}). Such monoids appear a lot in practice: commutative \'etale algebras in commutative algebra~\cite[Corollary 6.6]{Balmer}, \'etale extensions in algebraic geometry~\cite[Theorem 3.5 and Remark 3.8]{Balmer-Neeman-Thomasson}, $k(G/H)$ for subgroups $H<G$ of finite index in representation theory~\cite[Theorem 1.2]{Balmer-stacks}, $\Sigma^{\infty} (G/H)_+$ for $H$ a closed subgroup of finite index of a compact Lie group in equivariant stable homotopy theory~\cite[Theorem 1.1]{Balmer-equivariant}, and in other equivariant settings~\cite[Theorems 1.2, 1.3]{Balmer-equivariant}. More generally, in~\cite{Balmer} the author proves this result for modules over a separable exact monad (this includes Bousfield localization as a special case).

Stable derivators can be viewed as an enhancement of triangulated categories. One therefore might ask, given a separable exact monad $M$ on a stable derivator $\mathbb D$, whether the levelwise modules over $\mathbb D$ again form a stable derivator. We do show that this is indeed the case for derivators defined over "small" diagrams; for the general case we only need the extra assumption that the monad is smashing. In particular, this gives a new proof that modules over a separable monoid in a tensor triangulated category $\C$ inherit a natural triangulation from that of $\C$ if $\C$ is the base of a triangulated monoidal derivator $\mathbb D$ (which happens a lot in practice). As an application, we construct examples of derivators that are stable but not strong\footnote{In the literature, the axioms for a stable derivator include strongness; we depart from this by calling a derivator stable if it satisfies all axioms in the literature \textit{except} strongness (see Definitions~\ref{defn: stable} and~\ref{defn: strong})} (see Remark~\ref{rem: non-strong stable derivators}).

\noindent
\textbf{Acknowledgements:} I want to thank my advisor, Paul Balmer, for his support and help. I would also like to thank Kevin Carlson, Moritz Groth and Martin Gallauer for their comments and suggestions. Special thanks to Ian Coley for his invaluable help in editing this paper.

\section{2-Monads and their modules}
For the definition and basics on $2$-categories, the reader is referred to \cite{Kelly-Ross} or \cite{Borceux1}. We will denote by $\mathbf{CAT}$ the $2$-"category" of all categories, and by \textbf{2-CAT}\underline{} the $2$-"category" which has $2$-categories as $0$-cells, $2$-functors as $1$-cells, and $2$-natural transformations as $2$-cells. We summarize some basics on monads and $2$-monads which can all be found in \cite{Kelly-Ross} and \cite{Kelly}.

\begin{defn}Let $\K$ be a $2$-category, and let $x$ be an object of $\K$. A \textbf{monad} on $x$ consists of a triple~$(M, \mu, \mathbb S)$, where $M:x \to x$ is a $1$-cell in $\K$, and $\mu:M^2 \to M$, $\mathbb S:1_{x} \to M$ are $2$-cells such that the two diagrams below commute:
	\begin{equation} \label{eq: monad}
	\begin{tikzcd}
	M^3 \rar["\mu M"] \dar["M \mu"'] & M^2 \dar["\mu"] & M \rar["\mathbb S M"] \drar["1_M"'] & M^2 \dar["\mu"] & M \lar["M \mathbb S"'] \dlar["1_M"] \\
	M^2 \rar["\mu"'] & M &  & M &
	\end{tikzcd}
	\end{equation}
	We will often denote a monad just by $M$, leaving the multiplication and unit out of the notation.
\end{defn}

\begin{exam} \label{exam: Eilenberg-Moore}
	A classical monad is a monad in the $2$-category $\mathbf{CAT}$, see for instance \cite[Chapter VI]{MacLane}. Given a monad $M$ on a category $\C$, an $\mathbf{M}$\textbf{-module} (or $\mathbf{M}$\textbf{-algebra}) is a pair $(X, \lambda)$, where $X$ is an object of $\C$, and $\lambda:MX \to X$ is a morphism, such that the following diagrams commute:
		\begin{equation} \label{eq: module}
		\begin{tikzcd}
		M^2X \rar["M \lambda"] \dar["\mu_X"'] & MX \dar["\lambda"] & X \rar["\mathbb S_X"] \drar["1_X"'] & MX \dar["\lambda"]\\
		MX \rar["\lambda"'] & X &  & X
		\end{tikzcd}
		\end{equation}
	A \textbf{morphism of }$\mathbf{M}$\textbf{-modules} (or $\mathbf{M}$\textbf{-linear map}) $f:(X_1, \lambda_1) \to (X_2, \lambda_2)$ is a morphism $f:X_1 \to X_2$ in $\C$ such that the diagram below commutes:
	\begin{equation} \label{eq: linear map}
	\begin{tikzcd}
	MX_1 \rar["Mf"] \dar["\lambda_1"'] & MX_2 \dar["\lambda_2"] \\
	X_1 \rar["f"'] & X_2
	\end{tikzcd}
	\end{equation}
	The \textbf{Eilenberg-Moore category} $M\Mod_{\C}$ has as objects $M$-modules and as morphisms $M$-linear maps. There is also a functor $F_M: \C \to M \Mod_{\C}$ taking an object $X$ to $(MX, \mu_X)$ and a morphism $f$ to $Mf$. This functor is called the \textbf{free module functor}. There is also a forgetful functor $U_M:  M \Mod_{\C}  \to \C$ forgetting the action of $M$. We actually obtain an adjunction $F_M: \C \leftrightarrows M \Mod_{\C}:U_M$ called the \textbf{Eilenberg-Moore adjunction}. This is described in more detail in~\cite[Chapter VI]{MacLane}. In fact, this works more generally in any $2$-category by considering an action of the monad on $1$-cells instead of objects, see~\cite[Section 3]{Kelly-Ross}.
\end{exam}

\begin{defn}
	Given a $2$-category $\K$, a $\mathbf{2}$\textbf{-monad on }$\mathbf{\K}$ is a monad on $\K$ in the $2$-"category" \textbf{2-CAT}.
\end{defn}
Thus, given a $2$-category $\K$, a $2$-monad on $\K$ consists of a $2$-functor $T:\K \to \K$, together with $2$-natural transformations $\mu:T^2 \to T$ and $\mathbb S:1_{\K} \to T$ making the diagrams \eqref{eq: monad} commute. By forgetting the~$2$-structure of $T$ we get a classical monad $T_1$ on the underlying category $\K_1$ of $\K$. We can thus define~$\mathbf{T}$\textbf{-modules} as $T_1$-modules; a \textbf{strict }$\mathbf{T}$\textbf{-morphism} is a $T_1$-linear map (see Example~\ref{exam: Eilenberg-Moore}).

\begin{defn} \label{defn: lax D-morphism}
	Let $T$ be a $2$-monad on $\K$, and let $(x_1, \rho_1),(x_2, \rho_2)$ be $T$-modules. A \textbf{lax }$\mathbf{T}$\textbf{-morphism} from $(x_1, \rho_1)$ to $(x_2, \rho_2)$ consists of a pair $(f, \overline f)$, where $f:x_1 \to x_2$ is a $1$-cell, and $\overline f$ is a $2$-cell as in the following diagram:
	\[
	\begin{tikzcd}
	Tx_1 \rar["\rho_1"] \dar["Tf"']  & x_1 \dar["f"] \\
	Tx_2 \rar["\rho_2"']\urar[phantom, sloped, "\xRightarrow{\overline f}"] & x_2
	\end{tikzcd}
	\]
	such that we have an equality of pastings:
	\begin{equation}
	\begin{tikzcd}
	T^2x_1 \rar["\mu_{x_1}"] \dar["T^2f"'] & Tx_1 \rar["\rho_1"] \dar["Tf"'] & x_1 \dar["f"'] & T^2x_1 \rar["T \rho_1"] \dar["T^2f"] & Tx_1 \rar["\rho_1"] \dar["Tf"] & x_1 \dar["f"] \\
	T^2x_2 \rar["\mu_{x_2}"'] & Tx_2 \rar["\rho_2"'] \urar[phantom, sloped, "\xRightarrow{\overline f}"] & x_2 \urar[phantom,"="']  &  T^2x_2 \urar[phantom, sloped, "\xRightarrow{T\overline f}"] \rar["T \rho_2"'] & Tx_2 \rar["\rho_2"'] \urar[phantom, sloped, "\xRightarrow{\overline f}"] & x_2
	\end{tikzcd}
	\end{equation}
	and the following pasting is the identity:
	\begin{equation}
	\begin{tikzcd}
	x_1 \rar["\mathbb S_{x_1}"] \dar["f"'] & Tx_1 \rar["\rho_1"] \dar["Tf"'] & x_1 \dar["f"'] \\
	x_2 \rar["\mathbb S_{x_2}"'] & Tx_2 \urar[phantom, sloped, "\xRightarrow{\overline f}"] \rar["\rho_2"'] & x_2
	\end{tikzcd}
	\end{equation}
	The morphism is called \textbf{strong} if $\overline f$ is invertible, and \textbf{strict} if $\overline f$ is the identity. The dual notion is that of an \textbf{oplax }$\mathbf{T}$\textbf{-morphism}.
\end{defn}

\begin{defn} \label{defn: D-2-cell}
	Let $T$ be a $2$-monad on $\K$, and let $(f, \overline f),(g, \overline g):(x_1, \rho_1) \to (x_2, \rho_2)$ be lax $T$-morphisms. A $\mathbf{T}$\textbf{-}$\mathbf{2}$\textbf{-cell} from $f$ to $g$ is a $2$-cell $\a:f \Rightarrow g$ in $\K$ such that the following $2$ pastings are equal:
	\begin{equation}
	\begin{tikzcd}
	Tx_1 \rar["\rho_1"] \dar[bend right=90,"Tf"',""{name=U}] \dar["Tg",""'{name=V}] \arrow[Rightarrow, from=U, to=V,"T\a"] & x_1 \drar[phantom,"="] \dar["g"] & Tx_1 \dar["Tf"] \rar["\rho_1"] & x_1 \dar[bend left=90,"g",""'{name=B}] \dar["f"',""{name=A}] \arrow[Rightarrow, from=A, to=B,"\a"]\\
	Tx_2 \urar[phantom, sloped, "\xRightarrow{\overline g}"] \rar["\rho_2"'] & x_2 & Tx_2 \urar[phantom, sloped, "\xRightarrow{\overline f}"] \rar["\rho_2"'] & x_2
	\end{tikzcd}
	\end{equation}
\end{defn}
Given a $2$-monad $T$ on $\K$, one can then form a $2$-category $T\Mod_{\K}^{\text{lax}}$ which has $T$-modules as $0$-cells, lax-$T$-morphisms as $1$-cells, and $T$-$2$-cells as $2$-cells. The following theorem is stated in~\cite[Theorem 1.5]{Kelly} for the $2$-"category" $\mathbf{CAT}$ but is true for any $2$-category:

\begin{thm} \label{thm: doctrinal right adjunction}
	Let $T$ be a $2$-monad on a $2$-category $\K$. Let $F=(f, \overline f):(x_1, \rho_1) \to (x_2,\rho_2)$ be a lax~$T$-morphism. Then~$F$ admits a right adjoint in $T\Mod_{\K}^{\text{lax}}$ if and only if $f$ admits a right adjoint $g$ in~$\K$ and~$F$ is strong. Furthermore, the lax enrichment $\overline g$ of $g$ is given by the mate of $\overline f^{-1}$, namely as the pasting:
	\begin{equation}
		\begin{tikzcd}
			Tx_2 \ar[r,"Tg"] \ar[dr, bend right,""{name=A}, "1_{Tx_{2}}"'] \arrow[phantom, sloped, from=1-2, to=A, "\xLeftarrow{T\e}"'] & Tx_1 \ar[r,"\rho_1"] \ar[d,"Tf"] & x_1 \ar[d,"f"] \ar[phantom, sloped, dl, "\xLeftarrow{\overline f^{-1}}"] \ar[dr, bend left, "1_{x_1}", ""{name=B}"] \arrow[phantom, sloped, from=2-3, to=B, "\xLeftarrow{\eta}"] & \\
			& Tx_2 \ar[r, "\rho_2"'] & x_2 \ar[r,"g"'] & x_1
		\end{tikzcd}
	\end{equation}
\end{thm}

We will also need \cite[Theorem 1.4]{Kelly}, which is again true in any $2$-category:
\begin{thm} \label{thm: doctrinal left adjoint}
	Let $T$ be a $2$-monad on a $2$-category $\K$. Let $U=(g, \overline g)$ be a lax $T$-morphism. Then $U$ has a left adjoint $F$ in $T\Mod_{\K}^{\text{lax}}$ if and only if $g$ has a left adjoint $f$ in $\K$, and the mate $\overline f'$ of $\overline g$ given as the pasting:
	\begin{equation}
		\begin{tikzcd}
			& Tx_1 \ar[r, "\rho_1"] \ar[dr, phantom, sloped,"\xRightarrow{\overline g}"] & x_1 \ar[r, "f"]  & x_2 \\
			Tx_1 \ar[r, "Tf"'] \ar[ur, bend left, "1_{Tx_{1}}",""'{name=A}] \arrow[phantom, sloped, from=2-2, to=A, "\xRightarrow{T \eta}"] & Tx_2 \ar[r,"\rho_2"'] \ar[u,"Tg"'] & x_2 \ar[u, "g"] \ar[ur,bend right,""{name=B}, "1_{x_2}"'] & \arrow[phantom, sloped, from=1-3, to=B, "\xRightarrow{\e}"]
		\end{tikzcd}
	\end{equation}
	is an isomorphism. Then $F=(f, \overline f)$ is necessarily strong, and $\overline f=\overline f'^{-1}$.
\end{thm}
\begin{exam} \label{exam: 2-functors as algebras over a 2-monad}
	The $2$-category $\K=2\text{-}\mathbf{CAT}$ has all coproducts and in fact taking coproducts is a $2$-functor $\K \times \K \to \K$. The $2$-functor~$T:\K \times \K \to \K \times \K$ given by $(x,y) \mapsto (x, x \coprod y)$ is in fact a~$2$-monad; its multiplication and unit are given by the canonical $2$-functors $(x,x\coprod x \coprod y) \to (x,x\coprod y)$ and~$(x,y) \to (x,x \coprod y)$ respectively.
	
	A $T$-module consists of a pair $(x,y) \in \K \times \K$ together with a $1$-cell $(x, x \coprod y)=T(x,y) \to (x,y)$ in~$\K \times \K$ such that~\eqref{eq: module} commute (with $T$ for $M$ and $(x,y)$ for $X$). From this we see that the $1$-cells~$x \to x$ and~$y \to y$ are actually identities; hence a $T$-module is exactly a $1$-cell in $\K$ (a $2$-functor). Similarly, we see that lax~$T$-morphisms are exactly lax natural transformations, and $T\text{-}2\text{-}$cells are modifications. In short,~$T\Mod_{\K}^{\text{lax}}$ is the $2$-category of (strict) $2$-functors, lax natural transformations and modifications.
\end{exam}

\section{2-functoriality of levelwise modules}
Let us recall that the algebraic simplex category or augmented simplex category $\Delta_+$ has objects finite ordinal numbers $\underline n=\{0,1,...,n-1\}$ (where $\underline 0= \emptyset$), and morphisms order-preserving maps. This can be made in a strict monoidal category under the usual ordinal addition (see \cite[VII.5]{MacLane}). Thus, we can consider $\underline{\Delta_+}$, the $2$-category on one object with endomorphism category $\Delta_+$ and horizontal composition given by ordinal addition. For any $2$-category $\K$, let~$\text{Mon}_{\text{lax}}(\K)= \mathbf{2}\mathbf{-CAT}_{\text{lax}}(\underline{\Delta_+}, \K)$ be the $2$-category of $2$-functors $\underline{\Delta_+} \to \K$, lax natural transformations and modifications.

\begin{defn} \label{defn: monadic}
	Consider the $2$-"category" $\K=\text{Mon}_{\text{lax}}(\mathbf{CAT})$. Its $0$-cells will be called \textbf{monadic categories}, its $1$-cells \textbf{lax monadic functors}, and its $2$-cells \textbf{monadic natural transformations}.
\end{defn}
\noindent Let us spell out what this $2$-category actually is:
\begin{enumerate}
	\item Its $0$-cells are just $2$-functors $\underline{\Delta_+} \to \mathbf{CAT}$. Let $\C$ be the category that is the image of the unique~$0$-cell in $\underline{\Delta_+}$ under this $2$-functor. The endomorphism category $\underline{End}(\C)$ is a strict monoidal category under composition; the above being a $2$-functor just means we have a strict monoidal functor~$\Delta_+ \to \underline{End}(\C)$. By~\cite[Proposition VII.5.1]{MacLane} this corresponds exactly to a monad on~$\C$. Hence,~$0\text{-cells}$ are pairs $(\C,M)$ of a category with a monad on it.
	\item Similarly, we see that $1$-cells $(\C,M) \to (\C',M')$ are given by pairs $(F, \phi)$ where $F:\C \to \C'$ is a functor and $\phi: M'F \Rightarrow FM$ is a natural transformation such that the diagrams below commute:
	\begin{equation} \label{eq: monadic functor}
	\begin{tikzcd}
	M'^2F \rar["M' \phi"] \dar["\mu' F"'] & M'FM \rar["\phi M"]  & FM^2 \dar["F \mu"] & F \rar["\mathbb S' F"] \drar["F \mathbb S"'] & M'F \dar["\phi"]\\
	M'F \ar[rr,"\phi"']  & & FM & & FM
	\end{tikzcd}
	\end{equation}
	\item Its $2$-cells $(F_1, \phi_1) \Rightarrow (F_2, \phi_2):(\C,M) \to (\C',M')$ are just natural transformations $\a:F_1 \Rightarrow F_2$ such that the diagram below commutes:
	\begin{equation} \label{eq: monadic natural transformation}
	\begin{tikzcd}
	M'F_1 \ar[r, "M' \a"] \ar[d,"\phi_1"'] & M'F_2 \ar[d, "\phi_2"] \\
	F_1M \ar[r,"\a M"'] & F_2M
	\end{tikzcd}
	\end{equation}
	\item The various compositions and identities are the obvious ones.
\end{enumerate}
\begin{defn} \label{defn: oplax, strong and strict monadic functors}
	 An \textbf{oplax monadic functor} is a $1$-cell in $\text{Mon}_{\text{oplax}}(\K)= \mathbf{2}\mathbf{-CAT}_{\text{oplax}}(\underline{\Delta_+}, \K)$, the~$2$-category of $2$-functors $\underline{\Delta_+} \to \K$, oplax natural transformations and modifications. An (op)lax monadic functor whose structure $2$-cell is invertible will be called \textbf{strong}; if it is the identity it will be called \textbf{strict}.
\end{defn}
\begin{prop} \label{prop: 2-functor sending a monadic category to its modules}
	The assignment~$(\C,M)~\mapsto~M\Mod_{\C}$ extends to a $2$-functor~$\mathbf{Mod}:\operatorname{Mon}_{\operatorname{lax}}(\mathbf{CAT})~\to~\mathbf{CAT}$.
\end{prop}
\begin{proof} First, we define $\mathbf{Mod}$ on $1$-cells: Let $(F, \phi): (\C,M) \to (\C',M')$ be a monadic functor and let~$(X, \lambda) \in M \Mod_{\C}$. Note that the two diagrams below commute:
	\begin{equation*}
	\begin{tikzcd}
	M'^2FX\ar[dd, "\mu'_{FX}"'] \rar["M' \phi_X"] &
	M'FMX \rar["M'F \lambda"] \dar["\phi_{MX}"] &
	M'FX \dar["\phi_X"]  & & 		FX \rar["\mathbb S'_{FX}"] \ar[dr, "F \mathbb S_X"'] \ar[ddr, bend right=50, "1_{FX}"'] &
	M'FX \dar["\phi_X"]  \\
	&
	FM^2X \rar["FM \lambda"] \dar["F \mu_X"] &
	FMX \dar["F \lambda"] & & & FMX \dar["F \lambda"] \\
	M'FX \rar["\phi_X"] &
	FMX \rar["F \lambda"] &
	FX & & & FX
	\end{tikzcd}
	\end{equation*}
	(for the left diagram: left square by~\eqref{eq: monadic functor}, top right square by naturality of $\phi$, and bottom right square because $\lambda$ makes $X$ into an $M$-module; for the right diagram: top triangle by~\eqref{eq: monadic functor}, bottom because $\lambda$ makes $X$ into an $M$-module). This shows that $(FX, F \lambda \circ \phi_X)$ is an $M'$-module. Furthermore, given an~$M$-linear map $f:(X_1, \lambda_1) \to (X_2, \lambda_2)$ the following diagram commutes:
	\begin{equation*}
	\begin{tikzcd}
	M'FX_1 \rar["M'Ff"] \dar["\phi_{X_1}"] &
	M'FX_2 \dar["\phi_{X_2}"] \\
	FMX_1 \rar["FMf"] \dar["F \lambda_1"] &
	FMX_2 \dar["F \lambda_2"] \\
	FX_1 \rar["Ff"] &
	FX_2
	\end{tikzcd}
	\end{equation*}
	(top square by naturality of $\phi$ and bottom because $f$ is $M$-linear). Hence, we get a well-defined functor~$\overline F:~M\Mod_{\C}~\to~M'\Mod_{\C'}$ given on objects by $(X, \lambda) \mapsto (FX, F \lambda \circ \phi_X)$ and on morphisms as the identity. We define $\mathbf{Mod}(F, \phi)= \overline F$ on $1$-cells.
	
	We define $\mathbf{Mod}$ on $2$-cells: Let $\a:(F_1, \phi_1) \to (F_2, \phi_2)$ a monadic natural transformation between lax monadic functors $(\C,M) \to (\C',M')$. It is then straightforward from~\eqref{eq: monadic natural transformation} and naturality of $\a$ that for any $M$-module $(X, \lambda)$, $a_X$ is actually an $M'$-linear map \[(F_1X, F_1 \lambda \circ (\phi_1)_X) \to (F_2X, F_2 \lambda \circ (\phi_2)_X) \]
	and hence $\beta_{(X, \lambda)}=\a_X$ defines a natural transformation $\mathbf{Mod}(F_1, \phi_1) \to \mathbf{Mod}(F_2,\phi_2)$.
	
	Verification of $2$-functoriality is straightforward and omitted.
\end{proof}

\begin{rem} Let $\C, \C'$ be two monoidal categories and let $F: \C \to \C'$ be a lax monoidal functor. If $A$ is a monoid in $\C$, then $A'=FA$ is canonically a monoid in $\C'$ and $F$ descends to a functor~$A\Mod_{\C}~\to~A'\Mod_{\C'}$.
\end{rem}

\begin{cor} \label{cor: monadic left adjoint}
	Let $(F, \phi):(\C,M) \to (\C',M')$ be a strong monadic functor. If $F$ admits a right adjoint, then so does $\mathbf{Mod}(F)$.
\end{cor}
\begin{proof}
	By Proposition~\ref{prop: 2-functor sending a monadic category to its modules}, since a $2$-functor sends adjunctions to adjunctions, it is enough to show $(F, \phi)$ admits a right adjoint in $\K=\text{Mon}_{\text{lax}}(\mathbf{CAT})$. By Example~\ref{exam: 2-functors as algebras over a 2-monad}, this $2$-category is equivalent to the~$2$-category of modules over a certain $2$-monad. Hence the result follows by Theorem~\ref{thm: doctrinal right adjunction}.
\end{proof}

\begin{cor} \label{cor: monadic right adjoint}
	Let $(G,\psi):(\C',M') \to (\C,M)$ be a lax monadic functor. Assume $G$ has a left adjoint~$F$, and that the mate $\phi'$ of $\psi$ given as the pasting:
	\begin{equation}
	\begin{tikzcd}
	& \C \ar[r, "M"] \ar[dr, phantom, sloped,"\xRightarrow{\psi}"] & \C \ar[r, "F"]  & \C' \\
	\C \ar[r, "M"'] \ar[ur, bend left, "1_{\C}",""'{name=A}] \arrow[phantom, sloped, from=2-2, to=A, "\xRightarrow{ \eta}"] & \C' \ar[r,"M'"'] \ar[u,"G"'] & \C' \ar[u, "G"] \ar[ur,bend right,""{name=B}, "1_{\C'}"'] & \arrow[phantom, sloped, from=1-3, to=B, "\xRightarrow{\e}"]
	\end{tikzcd}
	\end{equation}
	is invertible. Then $\mathbf{Mod}(G)$ has a left adjoint.
\end{cor}
\begin{proof}
By Proposition~\ref{prop: 2-functor sending a monadic category to its modules}, since a $2$-functor sends adjunctions to adjunctions, it is enough to show $(G, \psi)$ admits a left adjoint in $\K=\text{Mon}_{\text{lax}}(\mathbf{CAT})$. By Example~\ref{exam: 2-functors as algebras over a 2-monad}, this $2$-category is equivalent to the $2$-category of modules over a certain $2$-monad. Hence the result follows by Theorem~\ref{thm: doctrinal left adjoint}.
\end{proof}

We now turn our attention to prederivators: we refer to Appendix~\ref{section: basics} for some basics or to \cite{Groth} for more detail. From now on, $\mathbf{Dia}$ will be a fixed diagram category (see definition \ref{defn: diagram category}). Unless specified otherwise, all (pre)derivators will be of domain~$\mathbf{Dia}$. There is a $2$-category of prederivators $\mathbf{PDer}$ with morphisms pseudonatural transformations, and $2$-morphisms modifications (see Appendix~\ref{section: 2-stuff} for the precise definition of pseudonatural transformations and modifications in this context or~\cite[Chapter 2]{Groth} for more details). In particular, we can consider monads in the $2$-category $\mathbf{PDer}$. A monad on a prederivator~$\mathbb D$ will consist of an endomorphism $M$ on $\mathbb D$ together with a modification $\mu:M^2 \to M$ and a modification~$\mathbb S:1_{\mathbb D} \to M$ such that the diagrams~\ref{eq: monad} commute. In particular, for each category~$J \in \mathbf{Dia}$, the triple~$(M_j, \mu_J, \mathbb S_J)$ defines a monad on $\mathbb D(J)$. In the remainder of this section we show that the~$M_J \Mod_{\mathbb D(J)}$ assemble to a prederivator, and the levelwise Eilenberg-Moore adjunctions ``glue'' to an adjunction in $\PDer$.

\begin{lem} \label{lem: lift to monadic categories}
	Let $M$ be a monad on a prederivator $\mathbb D$. Then $\mathbb D$ admits a lift $\widetilde{\mathbb D}$ given on $0$-cells by~$\mathbf{Dia}~\ni~J~\mapsto~(\mathbb D(J),~M_J)$ against the forgetful $2$-functor $U:\operatorname{Mon}_{\operatorname{lax}}(\mathbf{CAT}) \to \mathbf{CAT}$,\ie we have a factorization:
	\begin{equation*}
		\begin{tikzcd}
			& \operatorname{Mon}_{\operatorname{lax}}(\mathbf{CAT}) \ar[d, "U"] \\
		\mathbf{Dia}^{\text{op}} \ar[ru, dotted,"\widetilde{\mathbb D}"] \ar[r,"\mathbb D"'] & \mathbf{CAT}
		\end{tikzcd}
	\end{equation*}
\end{lem}
\begin{proof}
	Define $\widetilde{\mathbb D}$ on $1$-cells by~$u \mapsto (u^*, (\gamma^M_u)^{-1})$ (where $\gamma^M_u$ is the coherence isomorphism of the pseudonatural transformation $M$, see Appendix~\ref{section: 2-stuff}), and as the identity on $2$-cells.
	
	Given a functor $u:J \to K$ in $\mathbf{Dia}$, it follows that $(u^*, (\gamma^M_u)^{-1})$ is a lax (even strong) monadic functor by~\eqref{eq: modification} for $\mu$ and $\mathbb S$; and given a natural transformation as in
	\begin{equation*}
	\begin{tikzcd}
	J \rar[bend left=40, "u",""'{name=U}] \rar[bend right=40, "v"',""{name=V}] & K \arrow[Rightarrow, from=U, to=V, "\a"]
	\end{tikzcd}
	\end{equation*}
	it is clear by~\eqref{eq: 2-coherence} for $M$, that $\a^*$ is actually a monadic natural transformation from~$(u^*, (\gamma^M_u)^{-1})$ to~$(v^*, (\gamma^M_v)^{-1})$. Hence $\widetilde{\mathbb D}$ is actually well-defined, and its $2$-functoriality is immediately verified.
	 \end{proof}

\begin{prop} \label{prop: prederivator}
Let $M$ be a monad on a prederivator $\mathbb D$. The assignment:
\begin{equation*}
\mathbf{Dia}  \ni J \mapsto M_J\Mod_{\mathbb D(J)}
\end{equation*}
extends to a prederivator $M\Mod_{\mathbb D}: \mathbf{Dia}^{op} \to \mathbf{CAT}$. Furthermore, the free module functors~$F_{M,J}:\mathbb D(J)~\to~M_J\Mod_{\mathbb D(J)}$ glue together to a morphism of prederivators $F_M:\mathbb D \to M\Mod_{\mathbb D}$ and the forgetful functors $U_{M,J}:M_J\Mod_{\mathbb D(J)} \to \mathbb D(J)$ glue to a strict morphism of prederivators $U_M: M\Mod_{\mathbb D} \to \mathbb D$ such that $F_M$ is left adjoint to $U_M$ in the $2$-category $\mathbf{PDer}$ and $M=U_MF_M$.
\end{prop}
\begin{proof}
For the first part of the proof, note that by the previous lemma, we can factor $\mathbb D$ as
\begin{equation*}
\begin{tikzcd}
& \text{Mon}_{\text{lax}}(\mathbf{CAT}) \ar[d, "U"] \\
\mathbf{Dia}^{\text{op}} \ar[ru, dotted, "\widetilde{\mathbb D}"] \ar[r,"\mathbb D"'] & \mathbf{CAT}
\end{tikzcd}
\end{equation*}
and the result follows by postcomposing $\widetilde{\mathbb D}$ with the $2$-functor of Proposition~\ref{prop: 2-functor sending a monadic category to its modules}.

For the second part of the proof, let us write $\overline{u^*}=\mathbf{Mod}(u^*, (\gamma^M_u)^{-1})$ for a functor $u:J \to K$ in $\mathbf{Dia}$. Given such $u$, the diagram:
\begin{equation*}
\begin{tikzcd}
\mathbb D(J) & \mathbb D(K) \lar["u^*"] \\
M_J\Mod_{\mathbb D(J)} \uar["U_{M,J}"] & M_K\Mod_{\mathbb D(K)} \lar["\overline{u^*}"] \uar["U_{M,K}"']
\end{tikzcd}
\end{equation*}
commutes, which shows that $U_M$ is an honest $2$-natural transformation (instead of just pseudonatural).
We now want to define a natural isomorphism $\gamma^{F_M}_u$ populating the square:
\begin{equation*}
\begin{tikzcd}
M_J\Mod_{\mathbb D(J)}  & M_K\Mod_{\mathbb D(K)} \lar["\overline{u^*}"'] \dlar[phantom, sloped, "\xLeftarrow{\gamma^{F_M}_u}"] \\
\mathbb D(J) \uar["F_{M,J}"] & \mathbb D(K) \lar["u^*"] \uar["F_{M,K}"']
\end{tikzcd}
\end{equation*}
Given $X \in \mathbb D(K)$, we note that the $M_J$-module $\overline{u^*}F_{M,K}X$ has underlying object $u^*M_KX$ with action
\begin{equation*}
\lambda:M_Ju^*M_KX \xrightarrow{(\gamma^M_u)^{-1}_{M_KX}} u^*M_K^2X \xrightarrow{u^*(\mu_K)_X} u^*M_KX
\end{equation*}
while $F_{M,J}u^*X$ has underlying object $M_Ju^*X$ with action
\begin{equation*}
(\mu_J)_{u^*X}:M_J^2u^*X \xrightarrow{} M_Ju^*X
\end{equation*} 
Thus, taking $\gamma^{F_M}_u:=\gamma^M_u$, we need to show that $(\gamma^M_u)_X$ is $M_J$-linear. That is, we need to show that the following diagram commutes:
\begin{equation*}
\begin{tikzcd}
M_Ju^*M_KX \ar[rr, bend left,"\lambda"] \dar["M_J (\gamma^M_u)_X"'] &[+25pt] u^*M_K^2X \rar["u^*(\mu_K)_X"]  \dlar["(\gamma^{M^2}_u)_X"] \lar["(\gamma^M_u)_{M_KX}"',"\cong"] &[+25pt] u^*M_KX \dar["(\gamma^M_u)_X"] \\
M_J^2u^*X \ar[rr, "(\mu_J)_{u^*X}"'] &&  M_Ju^*X
\end{tikzcd}
\end{equation*}
But the triangle on the left commutes by the very definition of horizontal composite of pseudonatural transformations while the right square commutes because of diagram~\eqref{eq: modification} for $\mu$. The coherence axioms now follow directly from the fact that $\gamma^M_u$ satisfies them, so $F_M$ becomes a morphism of prederivators with structure maps those of $M$.

From the definitions, it is clear that $M=U_MF_M$. Hence, to finish, we need to show that the levelwise counits and units of the Eilenberg-Moore adjunctions assemble to modifications between the corresponding prederivators. Since prederivators are $2$-functors and these are modules over a certain $2$-monad (see Example~\ref{exam: 2-functors as algebras over a 2-monad}) it follows by Theorem~\ref{thm: doctrinal left adjoint} that $F_M$ is left adjoint to $U_M$ in the $2$-category of prederivators, \textit{lax} natural transformations and modifications. Since $\mathbf{PDer}$ is a $2$-full subcategory of the latter containing both $F_M$ and $U_M$ the result is immediate.
\end{proof}

\section{The left derivator of levelwise modules}
\noindent
Throughout this section $M$ will be a fixed monad on a fixed prederivator $\mathbb D$.
\begin{rem}We adopt the following notational conventions: given a functor $u:J \to K$ in $\mathbf{Dia}$, we will use the notation $u^*$ to refer to its pullback with respect to $\mathbb D$ and $u_!,u_*$ for the left and right adjoints respectively of $u^*$ (when they exist). We will denote the pullback with respect to $M\Mod_{\mathbb D}$ with $\overline{u^*}$ and its left and right adjoints respectively by $\overline{u_!},\overline{u_*}$ (when they exist). A similar convention will be used for natural transformations.
\end{rem}
\begin{prop} \label{prop: first two axioms}
Assume $\mathbb D$ satisfies $Der1$ and $Der2$. Then so does $M\Mod_{\mathbb D}$.
\end{prop}
\begin{proof}
For $Der1$: Consider a finite family $\{ J_i \}_{i\in I}$ of categories in $\mathbf{Dia}$. Let $J= \coprod J_i$ and for each $i$, let~$j_i: J_i \to J$ the canonical inclusions, and $\pi_i: \prod \mathbb D(J_i) \to \mathbb D(J_i)$ the canonical projections. For each index~$i$, we have a strong monadic functor $(j_i^*, \gamma^M_{j_i}):(\mathbb D(J),M_J) \to (\mathbb D(J_i), M_{J_i})$.

Note that $\prod \mathbb D(J_i)$ is actually a $2$-product, meaning that for any category $\K$, the induced functor
\begin{equation*}
\textbf{CAT}\left(\K, \prod \mathbb D(J_i)\right) \to \prod \textbf{CAT}(\K, \mathbb D(J_i))
\end{equation*}
(which is component-wise post-composition with $\pi_i$) is an isomorphism of categories. Therefore, the $j_i^*$ induce a unique functor $j^*: \mathbb D(J) \to \prod \mathbb D(J_i)$ (by abuse of the * notation) 
such that $\pi_i j^*= j_i^*$. It is then easy to see that the corresponding induced functor of $\{ M_{J_i}j_i^* \}$ is exactly $(\prod M_{J_i})j^*$ and the corresponding induced functor of $\{ j_i^*M_{J} \}$ is exactly $j^* M_{J}$. Therefore by the $2$-universal property above, the $\gamma^M_{J_i}$ induce a natural isomorphism $\phi: j^* M_{J} \overset{\sim}\rightarrow (\prod M_{J_i})j^*$. Since the $\gamma^M_{J_i}$ make the diagrams~\eqref{eq: monadic functor} commute (with $M_{J_i}$ for $M$), by uniqueness of the induced natural isomorphism, $\phi$ also has to make these diagrams commute; that is, $(j^*,\phi)$ is a strong monadic functor.

By $Der1$ for $\mathbb D$, the functor $j^*$ is an equivalence. Since $(j^*, \phi)$ is a strong monadic functor, Theorem~\ref{thm: doctrinal left adjoint} guarantees this is an equivalence $(\mathbb D(J),M_J) \xrightarrow{\cong} (\prod \mathbb D(J_i), \prod M_{J_i})$ in~$\text{Mon}_{\text{lax}}(\mathbf{CAT})$. Since $2$-functors preserve equivalences, by Proposition~\ref{prop: 2-functor sending a monadic category to its modules}  this induces a further equivalence $\overline {j^*}: M\Mod_{\mathbb D}(J) \xrightarrow{\cong} \prod M\Mod_{\mathbb D}(J_i)$. Finally, $\overline {j^*}$ is actually the functor induced by $\left\{ \overline{j_i^*} \right\}$ so $M\Mod_{\mathbb D}$ satisfies $Der1$.

For $Der2$, let $J$ be any category in $\mathbf{Dia}$, and $f:X \to Y$ a morphism in $M\Mod_{\mathbb D}(J)$. Then $f$ is an isomorphism if and only if $U_{M,J}f$ is an isomorphism (where $U_{M,J}:M_J\Mod_{\mathbb D(J)} \to \mathbb D(J)$ is the forgetful functor) if and only if (by $Der2$ for $\mathbb D$) $m^*U_{M,J}f \cong U_{M,e}\overline{m^*}f$ is an isomorphism for all $m \in J$ if and only if $\overline{m^*}f$ is an isomorphism for all $m \in J$ (again here using that the forgetful functor detects isomorphisms).
\end{proof}

Assume $\mathbb D$ is a left derivator (see Definition~\ref{defn: left derivator}). For any functor $u:J \to K$ in $\mathbf{Dia}$, the mate of~$\gamma^M_u: u^*M_K \to M_J u^*$ is the natural transformation
\begin{equation*}
\gamma^M_{u,*}:M_Ku_* \to u_*M_J
\end{equation*}
given as the pasting
\begin{equation*}
\begin{tikzcd}
\mathbb D(K) & \mathbb D(J) \lar["u_*"'] & \mathbb D(J) \lar["M_J"'] \dlar[phantom, sloped, "\xRightarrow{\gamma^M_u}"] &  \\
& \mathbb D(K) \uar["u^*"] \ar[ul, bend left, "1_{\mathbb D(K)}"{name=A}]  & \mathbb D(K) \uar["u^*"'] \lar["M_K"] & \mathbb D(J) \lar["u_*"] \ular[bend right, "1_{\mathbb D(J)}"'{name=C}] \arrow[phantom, sloped, from=C, to=2-3, "\xRightarrow{\e}"] \arrow[phantom, sloped, from=A, to=1-2, "\xRightarrow{\eta}"]
\end{tikzcd}
\end{equation*}

\begin{prop} \label{prop: left}
If $\mathbb D$ is a left derivator, then the prederivator $M\Mod_{\mathbb D}$ is also a left derivator.
\end{prop}
\begin{proof}
Let us fix a functor $u:J \to K$ in $\mathbf{Dia}$. By~\eqref{eq: modification} for $\mu$ and $\mathbb S$, $(u^*, \gamma^M_u)$ is a strong monadic functor. Since $u^*$ has a right adjoint $u_*$, by Corollary~\ref{cor: monadic left adjoint} we get an induced adjunction $\overline{u^*}: M\Mod_{\mathbb D}(K) \leftrightarrows M\Mod_{\mathbb D}(J): \overline{u_*}$ proving $Der3L$ for $M\Mod_{\mathbb D}$.

\noindent
For $Der4L$, consider any functor $u:J \to K$ in $\mathbf{Dia}$, and let $k \in K$. We have a diagram in $\mathbf{Cat}$:
\begin{equation*}
\begin{tikzcd}
J_{k/} \rar["pr"] \dar["\pi"]  &
J \dar["u"] \ar[ld, phantom, sloped, "\xRightarrow{\a}"] \\
e \rar["k"] &
K
\end{tikzcd}
\end{equation*}

\noindent
By assumption the mate $\a_*:k^*u_* \to \pi_* pr^*$ of $\a^*$ given as the pasting
\begin{equation*}
\begin{tikzcd}
\mathbb D(e) &
\mathbb D(J_{k/}) \lar["\pi_*"']  &
\mathbb D(J) \lar["pr^*"']  \ar[ld, phantom, sloped, "\xRightarrow{\a^*}"] & \\
&
\mathbb D(e) \ar[lu, bend left, ""'{name=A}, "1_{\mathbb D(e)}"]  \uar["\pi^*"']  &
\mathbb D(K) \uar["u^*"] \lar["k^*"] &
\mathbb D(J) \lar["u_*"] \ar[ul, bend right, "1_{\mathbb D(J)}"', ""{name=B}]
\ar[phantom, sloped, from=A, to=1-2, "\xRightarrow{\eta}"] \ar[phantom, sloped, from=2-3, to=B, "\xRightarrow{\e}"]
\end{tikzcd}
\end{equation*}
is an isomorphism. Given $(X, \lambda) \in M\Mod_{\mathbb D}(J)$ we clearly have $U_{M,e}((\overline{\a_*})_{(X, \lambda)})=(\a_{*})_{X}$. Since the forgetful functor detects isomorphisms, we are done.
\end{proof}

\begin{rem}
	Dually, given a comonad $M$ on a right derivator $\mathbb D$, the levelwise comodules over $M$ form a right derivator.
\end{rem}

\section{Left Kan Extensions}
Throughout this section, let $M$ be a fixed monad on a prederivator $\mathbb D$. For any functor $u:J \to K$ in~$\mathbf{Dia}$ such that $\mathbb D$ admits left Kan extensions along $u$, we have a canonical mate:
\begin{equation*}
\gamma^M_{u,!}:u_!  M_J \to M_Ku_!
\end{equation*}
given as the pasting: 
\begin{equation} \label{eq: gammalowershriek}
\begin{tikzcd}
\mathbb D(K)  & \mathbb D(J) \lar["u_!"'] & \mathbb D(J) \lar["M_J"'] \dlar[phantom, sloped, "\xLeftarrow{(\gamma^M_u)^{-1}}"] & \arrow[phantom, sloped, from=1-2, to=A,"\xLeftarrow{\e}"] \\
& \ular[bend left, "1_{\mathbb D(K)}"{name=A}] \mathbb D(K) \uar["u^*"]  & \mathbb D(K) \uar["u^*"'] \lar["M_K"] & \mathbb D(J) \lar["u_!"] \ular[ bend right, "1_{\mathbb D(J)}"'{name=C}] \arrow[phantom, sloped, from=C, to=2-3, "\xLeftarrow{\eta}"]
\end{tikzcd}
\end{equation}
\begin{defn} \label{defn: cocontinuous}
	If this mate is an isomorphism, then we say $M$ \textbf{commutes with left Kan extensions along}~$\mathbf{u}$. If this happens for any functor $u$ in $\mathbf{Dia}$ then we say $M$ is cocontinuous.
\end{defn}
\begin{prop}
Let $u:J \to K$ be a functor in $\mathbf{Dia}$ such that $\mathbb D$ admits left Kan extensions along $u$. If~$M$ commutes with left Kan extensions along $u$ then $M\Mod_{\mathbb D}$ admits left Kan extensions along $u$.
\end{prop}
\begin{proof}
We have a lax monadic functor $U=(u^*, (\gamma^M_u)^{-1}): (\mathbb D(K),M_K) \to (\mathbb D(J), M_J)$. By Corollary~\ref{cor: monadic right adjoint},~$M$ commutes with left Kan extensions along $u$ if and only if $U$ has a lax adjoint, which is then necessarily strong. Applying the $2$-functor of Proposition~\ref{prop: 2-functor sending a monadic category to its modules} the result follows immediately.
\end{proof}


\begin{cor}
	Let $\mathbb D$ be a right derivator. If $M$ is cocontinuous, then $M\Mod_{\mathbb D}$ is also a right derivator and the right adjoint of the Eilenberg-Moore adjunction (see Proposition~\ref{prop: prederivator}) is cocontinuous.
\end{cor}
\begin{proof}
	The last proposition guarantees $M\Mod_{\mathbb D}$ has all left Kan extensions, while Proposition~\ref{prop: first two axioms} guarantees it satisfies $Der1$ and $Der2$. The verification of $Der4L$ follows exactly as in Proposition~\ref{prop: left}.
	
	To finish the proof it remains to show that $U_M$ is cocontinuous: let $u:J \to K$ a functor in $\mathbf{Dia}$. Let us write $\e$ for the counit of the adjunction $u_! \dashv u^*$ and $\overline \eta$ for the unit of the adjunction $\overline{u_!} \dashv \overline{u^*}$.  We have to show that the top arrow in the following commutative diagram is an isomorphism:
	\[
	\begin{tikzcd}
	u_! U_{M,J} \ar[r] \ar[d,"u_! U_{M,J} \overline{\eta}"'] & U_{M,K} \overline{u_!} \\
	u_!U_{M,J}\overline{u^*}\overline{u_!} \ar[r, equals] & u_!u^*U_{M,K}\overline{u_!} \ar[u,"\e U_{M,K}\overline{u_!}"'] 
	\end{tikzcd}
	\]
	But applied to an object $(X,\lambda) \in M_J\Mod_{\mathbb D(J)}$ this diagram becomes
	\[
	\begin{tikzcd}
	u_!X \ar[d, "u_! \eta"'] \ar[r] & u_!X \\
	u_!u^*u_!X \ar[r, "1"'] & u_!u^*u_!X \ar[u,"\e u_!"]
	\end{tikzcd}
	\]
	so the top arrow is the identity by the triangle identities, finishing the proof.
\end{proof}

\begin{rem}
	Of course $F_M$ is also cocontinuous as a left adjoint morphism of derivators. So while Proposition~\ref{prop: left} says we can do levelwise modules in the $1$- and $2$-full subcategory of $\PDer$ spanned by left derivators, the above proposition says that for right derivators we have to restrict further to \textbf{cocontinuous} morphisms as $1$-cells for everything to work. A notable exception is when the monad in question is idempotent (see~\cite[Lemma 4.2]{Cisinski}).
\end{rem}

\begin{cor} \label{cor: main}
Assume $\mathbb D$ is a derivator, and $M$ a cocontinuous monad on $\mathbb D$. Then the prederivator $M\Mod_{\mathbb D}$ is also a derivator and the forgetful functor $U_M: M\Mod_{\mathbb D} \to \mathbb D$ of Proposition~\ref{prop: prederivator} is cocontinuous.
\end{cor}

\begin{rem} \label{rem: exact morphism}
	Recall that a morphism of derivators $F: \mathbb D \to \mathbb E$ is called \textbf{right exact} if it preserves initial objects and cocartesian squares (see Definition~\ref{defn: cocartesian}). When~$\mathbf{Dia}=\mathbf{Cat}$, any right exact monad on $\mathbb D$ that preserves coproducts is automatically cocontinuous (see \cite[Theorem 7.13]{linearity}).
\end{rem}
\section{Stability}
For this section, assume $\mathbb D$ is a derivator and $M$ cocontinuous, so that by Corollary~\ref{cor: main} above $M\Mod_{\mathbb D}$ is a derivator. As noted in Remark~\ref{rem: exact morphism} above, it is enough to assume~$M$ is right exact and commutes with coproducts if $\mathbf{Dia}=\mathbf{Cat}$.

By $\square$ we will denote the following category
\begin{equation*}
\begin{tikzcd}
(0,0) \rar[""] \dar[""] & (1,0) \dar[""] \\
(0,1) \rar[""] & (1,1)
\end{tikzcd}
\end{equation*}
and we will use the notations $\ulcorner, \lrcorner$ for the top left and bottom right corners respectively (\ie the full subcategories missing (1,1) or (0,0)). The corresponding inclusions will be denoted by $i_{\ulcorner}$ and $i_{\lrcorner}$. An object $X \in \mathbb D(\square)$ will be called a \textbf{(coherent) square}.
\begin{defn} \label{defn: cocartesian}
A square $X \in \mathbb D(\square)$ is called \textbf{cocartesian} if and only if it lies in the essential image of~${i_{\ulcorner}}_!:\mathbb D(\ulcorner)\to\mathbb D(\square)$. It is \textbf{cartesian} if and only if it lies in the essential image of ${i_{\lrcorner}}_*:\mathbb D(\lrcorner)\to\mathbb D(\square)$.
\end{defn}
Note that $i_{\ulcorner}, i_{\lrcorner}$ are fully faithful, and because homotopy left or right Kan extensions along fully faithful functors are also fully faithful (see~\cite[Proposition 1.20]{Groth}), we see that $X \in \mathbb D(\square)$ is cocartesian if and only if the counit $\e_X: {i_{\ulcorner}}_! {i_{\ulcorner}}^* X \to X$ is an isomorphism. Dually, $X$ is cartesian if and only if the unit~$\eta_X:X \to {i_{\lrcorner}}^* {i_{\lrcorner}}_*X$ is an isomorphism.
\begin{defn} \label{defn: stable}
A derivator $\mathbb D$ is \textbf{stable} if it is pointed (\ie the category $\mathbb D(J)$ is pointed for any small category $J$) and a square in $\mathbb D(\square)$ is cocartesian if and only if it is cartesian.
\end{defn}
\begin{prop} \label{prop: stable}
If $\mathbb D$ is a stable derivator, and $M$ is a cocontinuous monad on $\mathbb D$, then $M\Mod_{\mathbb D}$ is also stable. 
\end{prop}
\begin{proof}
Since $M$ is cocontinuous, $M\Mod_{\mathbb D}$ is a derivator. Furthermore, $M$ is in particular pointed and hence, $M\Mod_{\mathbb D}$ is also pointed. Let us write $\e$ for the counit of the adjunction ${i_{\ulcorner}}_! \dashv i_{\ulcorner}^*$, and $\overline{\e}$ for that of $\overline{{i_{\ulcorner}}_!} \dashv \overline{i_{\ulcorner}^*}$. Let $(X ,\lambda) \in M\Mod_{\mathbb D}( \square)$. Note that $U_M(\overline{\e}_{(X, \lambda)})=\e_X$ so $(X, \lambda)$ is cocartesian if and only if~$X \in \mathbb D( \square)$ is. Similarly $(X, \lambda)$ is cartesian if and only if $X$ is. The result follows immediately.
\end{proof}
\begin{rem} \label{rem: non-strong stable derivators}
	Proposition~\ref{prop: stable} allows us to produce examples of derivators that are stable but not strong. Indeed, let $\mathbb D$ be a stable strong monoidal derivator, and $A$ any monoid in~$\mathbb D(e)$; then~$A\Mod_{\mathbb D}$ is a stable derivator. If it is also strong, then its underlying category $A\Mod_{\mathbb D(e)}$ has to be triangulated. Furthermore, as in the proof of Corollary~\ref{cor: triangulated}, the morphism of derivators~$U:A\Mod_{\mathbb D} \to \mathbb D$ has to be exact. But then~\cite[Proposition 4.18]{Groth} guarantees that the forgetful functor~$U_e:A\Mod_{\mathbb D(e)} \to \mathbb D(e)$ can be (canonically) endowed with the structure of an exact functor of triangulated categories. In short, if the stable derivator $A\Mod_{\mathbb D}$ is also strong, then its underlying category~$A\Mod_{\mathbb D(e)}$ has to admit a triangulation so that the forgetful functor $A\Mod_{\mathbb D(e)} \to \mathbb D(e)$ is exact. Hence, the recipe for producing non-strong stable derivators is: take any tensor triangulated category $\C$ underlying a strong, stable, monoidal derivator $\mathbb D$, and any monoid object $A$ on the base such that its modules do not admit a triangulated structure compatible with that of $\C$. Then $A\Mod_{\mathbb D}$ is an example of a non-strong stable derivator. Such examples abound in nature but we give an explicit one below.
\end{rem}

\begin{exam} \label{exam: modules not compatible with triangulation}
	Let $k$ be a field, and let $A$ be any $k$-algebra which admits non-projective modules. Considering~$A$ as a complex concentrated in degree $0$, we get a monoid in $\mathbb D(k)$. Assume $\T:=A\Mod_{\mathbb D(k)}$ admits a triangulation such that the forgetful functor $\T \to D(k)$ is exact. Then for any $X \in \T$, the counit~$\e_X$ admits a section (see the proof of~\cite[Proposition 2.10]{Balmer}). Since $A$-modules are equivalent to graded $k$-vector spaces with a compatible action of~$A$, this implies that any $A$-module is projective, a contradiction.
\end{exam}
\section{Separable monads and strongness}
\begin{defn}
A monad $M: \mathbb D \to \mathbb D$ in the $2$-category $\PDer$ is \textbf{separable} if $\mu:M^2 \to M$ admits a section $\sigma:M \to M^2$ satisfying $M \mu \circ \sigma M= \sigma \circ \mu= \mu M \circ M \sigma$.
\end{defn}
\begin{lem} \label{lem: section}
Let $M: \mathbb D \to \mathbb D$ a monad on a prederivator $\mathbb D$. Then $M$ is separable if and only if the counit of the Eilenberg-Moore adjunction admits a section.
\end{lem}
\begin{proof}
Let $\xi:1_{M\Mod_{\mathbb D}} \to F_MU_M$ be a section of $\e:F_MU_M \to 1_{M\Mod_{\mathbb D}}$. Define $\sigma:M \to M^2$ as the composite $M=U_MF_M \xrightarrow{U_M \xi F_M} U_MF_MU_MF_M=M^2$, which is necessarily a modification. Then by~\cite[Proposition 6.3]{Virelizier}, $M$ is separable.

Conversely, assume $M$ is separable and let $\sigma: M \to M^2$ a section of $\mu$ satisfying $M \mu \circ \sigma M= \sigma \circ \mu= \mu M \circ M \sigma$. For any category $J$ in $\mathbf{Dia}$, and any $(X, \lambda) \in M_J\Mod_{\mathbb D(J)}$, define
$$(\xi_J)_{(X, \lambda)}: X \xrightarrow{\mathbb S_{J,X}} M_JX \xrightarrow{(\sigma_J)_X} M_J^2X \xrightarrow{M_J \lambda} M_JX$$
where $\mathbb S$ is the unit of the monad $M$, which is also the unit of the Eilenberg-Moore adjunction. It is proven in~\cite[Proposition 6.3]{Virelizier} that this defines a natural transformation $\xi_J: 1_{M\Mod_{\mathbb D}(J)} \to F_{M,J}U_{M,J}$ that is a section of the counit $\e_J$. All that is left for us is to actually show this is a modification. So let $u:J \to K$ a functor in $\mathbf{Dia}$. We need to show that the diagram:
\begin{equation*}
\begin{tikzcd}
u^* \rar["u^* \xi_K"] \drar["\xi_J u^*"'] & u^*F_{M,K}U_{M,K} \dar["\gamma^M_u"] \\
& F_{M,J}U_{M,J}u^*
\end{tikzcd}
\end{equation*}
commutes. So, let $(X, \lambda) \in M_K\Mod_{\mathbb D(K)}$. Then $u^*((\xi_{K})_{(X, \lambda)})$ is the composite:
\begin{equation*}
u^*X \xrightarrow{u^* ((\mathbb S_K)_X)} u^*M_KX \xrightarrow{u^* ((\sigma_K)_X)} u^*M_K^2X \xrightarrow{u^* M_K \lambda} u^*M_KX
\end{equation*}
and $(\xi_J)_{u^*X}$ is the composite:
\begin{equation*}
u^*X \xrightarrow{(\mathbb S_J)_{u^*X}} M_Ju^*X \xrightarrow{(\sigma_J)_{u^*X}} M_J^2u^*X \xrightarrow{M_J (\gamma^M_u)_X^{-1}} M_J u^* M_KX \xrightarrow{M_J u^* \lambda} M_J u^*X
\end{equation*}

\noindent
Now it suffices to note that the diagram
\begin{equation*}
\begin{tikzcd}[row sep = huge]
u^*X \rar["u^* (\mathbb S_K)_X"] \dar["(\mathbb S_J)_{u^*X}"'] &[+25pt]
u^*M_KX \rar["u^* (\sigma_K)_X"] \dlar["(\gamma^M_u)_X"'] &[+25pt]
u^*M_K^2X \rar["u^* M_K \lambda"] \dlar["(\gamma^{M^2}_u)_X"'] \dar["(\gamma^M_u)_{M_KX}"] &[+25pt]
u^*M_KX \dar["(\gamma^M_u)_X"] \\
M_Ju^*X \rar["(\sigma_J)_{u^*X}"'] &
M_J^2 u^*X \rar["M_J (\gamma^M_u)_X^{-1}"'] &
M_Ju^*M_KX \rar["M_J u^* \lambda"'] &
M_Ju^*X
\end{tikzcd}
\end{equation*}
commutes (left triangle by~\eqref{eq: modification} for $\mathbb S$, the other triangle by definition of composition of pseudonatural transformations, the skew parallelogram because $\sigma$ is a modification, and the right square by naturality of~$\gamma^M_u$).
\end{proof}

\begin{defn} \label{defn: strong}
A derivator $\mathbb D$ is \textbf{strong} if it satisfies the following axiom:

\noindent
($Der5$) For any category $J$ in $\mathbf{Dia}$, the partial underlying diagram functor $\dia_{J,[1]}:\mathbb D([1] \times J) \to \mathbb D(J)^{[1]}$ is full and essentially surjective (here $[1]$ stands for the poset $\{ 0<1\}$ viewed as a category).
\end{defn}

\begin{defn} \label{defn: additive}
	A derivator $\mathbb D$ is \textbf{additive} if its underlying category $\mathbb D(e)$ is additive. Equivalently, by \cite[Proposition 5.2]{Groth}, all its values are additive categories, and $\mathbb D$-pullbacks and Kan extensions are additive functors. An additive derivator is \textbf{idempotent complete} if its values are idempotent complete categories.
\end{defn}

\begin{lem} \label{lem: idempotent} 
Let $\mathbb D$ be an additive idempotent-complete derivator. Let $f:X \to Y$ a morphism in $\mathbb D(e)$ and $F \in \mathbb D([1])$ a lift of $f$ \ie $\dia_{[1]}(F) \cong f$. Assume furthermore that $(a,b):f \to f$ is an idempotent in~$\mathbb D(e)^{[1]}$ (\ie $a^2=a,b^2=b$). Then there is an idempotent $e:F \to F$ in $\mathbb D([1])$ lifting $(a,b)$.
\end{lem}
\begin{proof}
As $a$ is an idempotent of $X$, there is a splitting $X \cong A \oplus C$ with $a \cong \left ( \begin{smallmatrix} 1 & 0 \\ 0 & 0 \end{smallmatrix} \right ):A \oplus C \to A \oplus C$. Similarly, there is a splitting $Y \cong B \oplus D$ with $b \cong \left ( \begin{smallmatrix} 1 & 0 \\ 0 & 0 \end{smallmatrix} \right ): B \oplus D \to B \oplus D$. Let $f \cong \left ( \begin{smallmatrix} w & x \\ y & z \end{smallmatrix} \right ): A \oplus C \to B \oplus D$. We then have a commutative diagram
\begin{equation*} 
\xymatrix{
X \ar[dr]^{\cong} \ar[rrr]^f \ar[ddd]_a &&& Y \ar[ddd]^b \ar[dl]_{\cong} \\
& A \oplus C \ar[r]^{\left ( \begin{smallmatrix} w & x \\ y & z \end{smallmatrix} \right )} \ar[d]_{\left ( \begin{smallmatrix} 1 & 0 \\ 0 & 0 \end{smallmatrix} \right )} & B \oplus D \ar[d]^{\left ( \begin{smallmatrix} 1 & 0 \\ 0 & 0 \end{smallmatrix} \right )} & \\
& A \oplus C \ar[r]_{\left ( \begin{smallmatrix} w & x \\ y & z \end{smallmatrix} \right )}  & B \oplus D & \\
X \ar[ur]^{\cong} \ar[rrr]_f &&& Y \ar[ul]_{\cong}
}
\end{equation*}
from which we deduce that $x=0$ and $y=0$. It follows that $f \cong w \oplus z: A \oplus C \to B \oplus D$.  Pick a lift~$W \in \mathbb D([1])$ of $w$ (\ie $\dia_{[1]}(W) \cong w$) and a lift $Z \in \mathbb D([1])$ of $z$. Since pullback is an additive functor, we can see that $W \oplus Z$ lifts $f$. By strongness, the identity on $f$ lifts to a morphism $F \to W \oplus Z$ in~$\mathbb D([1])$. This is a pointwise isomorphism by construction, hence by $Der2$ it is an isomorphism. Then we can define an idempotent of $F$ in~$\mathbb D([1])$ via the map $F \cong W \oplus Z \xrightarrow{\left ( \begin{smallmatrix} 1 & 0 \\ 0 & 0 \end{smallmatrix} \right )} W \oplus Z \cong F$ which clearly lifts the idempotent~$(a,b)$ of $f$.
\end{proof}
\begin{rem} \label{rem:big implies idempotent complete}
	When $\mathbb D$ is a big derivator(see Definition~\ref{defn: big}) that is also stable and strong, then $\mathbb D(e)$ admits all small coproducts (see~\cite[Proposition 1.7]{Groth}) and is triangulated (see \cite[Proposition 4.16]{Groth}). Hence, by a well-known fact for triangulated categories, $\mathbb D(e)$ is idempotent-complete (see for instance~\cite[Proposition 1.6.8]{Neeman}).
\end{rem}
\begin{prop} \label{prop: strongness}
Let $F: \mathbb D \leftrightarrows \mathbb E : G$ an adjunction of derivators with $\mathbb D$ stable, strong and idempotent complete, $\mathbb E$ additive, and idempotent-complete. Furthermore, assume that the counit $\e:FG \to 1_{\mathbb E}$ admits a section $\xi: 1_{\mathbb E} \to FG$. Then $\mathbb E$ is strong.
\end{prop}
\begin{proof}
The assumptions are preserved under shifting, so enough to show that $\dia^{\mathbb E}_{[1]}: \mathbb E([1]) \to \mathbb E(e)^{[1]}$ is full and essentially surjective. 

For essential surjectivity, let $f:X \to Y$ a morphism in $\mathbb E(e)$. The morphism $g=G_ef:G_eX \to G_eY$ in $\mathbb D(e)$ admits a lift $\overline g \in \mathbb D([1])$. Set $\overline f=F_{[1]} \overline g$. Since $F$ is a morphism of derivators, it commutes with pullbacks (up to isomorphism). Hence by construction of the underlying diagram functor, it follows that~$\dia^{\mathbb E}_{[1]} F_{[1]} \cong F_e \dia^{\mathbb D}_{[1]}$. Hence, $\dia^{\mathbb E}_{[1]}\overline f \cong F_eg=F_eG_ef$. 

Note that again since $G$ is a morphism of derivators, we have $\dia^{\mathbb D}_{[1]}G_{[1]} \overline f \cong G_e \dia^{\mathbb E}_{[1]} \overline f \cong G_eF_eg$. Set~$r=\xi_e \e_e$. Note $(r_X,r_Y)$ is an idempotent of $F_eg$. Then $(G_er_X,G_er_Y)$ is an idempotent of $G_eF_eg$. By Lemma~\ref{lem: idempotent}, it follows that there is an idempotent $\tilde r$ of $G_{[1]} \overline f$ with $\dia^{\mathbb D}_{[1]}\tilde r \cong (G_er_X,G_er_Y)$.

Set $\overline r=\e_{\overline f} \circ F_{[1]}\tilde r \circ \xi_{\overline f}$, which is an endomorphism of $\overline f$, lifting $(r_X,r_Y)$. Since the last is an idempotent of $F_eG_ef$, it follows $a=\overline r^2- \overline r$ lifts the $0$ endomorphism of $F_eG_ef$. Hence by~\cite[p. 206, Theorem 3.8]{Grothdis} it follows that $a^2=0$. Finally, $b=3 \overline r^3-2 \overline r^2$ is an idempotent of $\overline f$ by direct computation, which lifts the idempotent $(r_X,r_Y)$ of $F_eG_ef$. 

So far, we showed that the idempotent $(r_X,r_Y)$ of $F_eG_ef$ lifts to an idempotent $\overline r$ of the lift $\overline f$ of $F_eG_ef$. Since by assumption $\mathbb E$ is idempotent complete, there is a splitting $\overline f \cong K \oplus H$ such that $\overline r$ corresponds to~$\left ( \begin{smallmatrix} 0 & 0 \\ 0 & 1 \end{smallmatrix} \right )$. Now using additivity of pullbacks in any additive derivator, it follows that the underlying diagram of $\overline r$:
\begin{equation*}
\begin{tikzcd}
F_eG_eX \rar["F_eG_ef"] \dar["r_X"'] &
F_eG_eY \dar["r_Y"] \\
F_eG_eX \rar["F_eG_ef"'] &
F_eG_eY
\end{tikzcd}
\end{equation*}
is isomorphic to a diagram of the form 
\begin{equation*}
\xymatrix{
\widetilde X \oplus \widetilde X' \ar[r]^{\left ( \begin{smallmatrix} a & 0 \\ 0& b \end{smallmatrix} \right )} \ar[d]_{\left ( \begin{smallmatrix} 0 & 0 \\ 0 & 1 \end{smallmatrix} \right )} &
\widetilde Y \oplus \widetilde Y' \ar[d]^{\left ( \begin{smallmatrix} 0 & 0 \\ 0 & 1 \end{smallmatrix} \right )} \\
\widetilde X \oplus \widetilde X' \ar[r]_{\left ( \begin{smallmatrix} a & 0 \\ 0& b \end{smallmatrix} \right )} &
\widetilde Y \oplus \widetilde Y'
}
\end{equation*}

It then follows easily from uniqueness of splitting of idempotents, that $f \cong b$ and hence $\dia^{\mathbb E}_{[1]} H \cong f$. This proves essential surjectivity.

For fullness, assume that we have a commutative diagram
$$\begin{tikzcd}
A \rar["x"] \dar["a"'] &
B \dar["b"] \\
C \rar["y"'] &
D
\end{tikzcd}$$
in $\mathbb D(e)$ and chosen lifts $X,Y$ of $x,y$ respectively in $\mathbb E([1])$. Then, applying strongness of $\mathbb D$ we deduce that there is a lift $\phi: G_{[1]}X \to G_{[1]}Y$ lifting $(G_ea,G_eb)$. Hence $F_{[1]} \phi$ lifts $(F_eG_ea, F_eG_eb)$.

Set~$\overline {\phi}=(\e_{[1]})_Y F_{[1]} \phi (\xi_{[1]})_X: X \to Y$. By the fact that $\xi, \e$ are modifications, and functoriality of underlying diagram, we deduce that the underlying diagram of $\overline {\phi}$ is isomorphic to
$$\begin{tikzcd}
A \rar["x"] \dar["(\xi_e)_A"'] & 
B \dar["(\xi_e)_B"] \\
F_eG_eA \rar["F_eG_ex"] \dar["F_eG_ea"'] &
F_eG_e B \dar["F_eG_eb"] \\
F_eG_eC \rar["F_eG_ey"] \dar["(\e_e)_C"'] &
F_eG_e D \dar["(\e_e)_D"] \\
C \rar["y"'] &
D
\end{tikzcd}$$
which by naturality of $\e_e$ is in turn isomorphic to
$$\begin{tikzcd}
A \rar["x"] \dar["(\xi_e)_A"'] & 
B \dar["(\xi_e)_B"] \\
F_eG_eA \rar["F_eG_ex"] \dar["(\e_e)_A"'] &
F_eG_e B \dar["(\e_e)_B"] \\
A \rar["x"] \dar["a"'] &
B \dar["b"] \\
C \rar["y"'] &
D
\end{tikzcd}$$
which is exactly the original diagram (since $\e_e \circ \xi_e=1_{1_{E(e)}}$).
\end{proof}
\begin{cor} \label{cor: triangulated}
Let $M: \mathbb D \to \mathbb D$ a cocontinuous separable monad on a stable, strong and idempotent-complete derivator $\mathbb D$. Then the prederivator $M\Mod_{\mathbb D}$ of Proposition~\ref{prop: prederivator} is a strong stable derivator. Moreover, the morphisms $F_M,U_M$ of Proposition~\ref{prop: prederivator} are exact.
\end{cor}
\begin{proof}
By Proposition~\ref{prop: stable}, $M\Mod_{\mathbb D}$ is a stable derivator. The monad $M$ is realized by the adjunction~$F_M: \mathbb D \leftrightarrows M\Mod_{\mathbb D}: U_M$ of Proposition~\ref{prop: prederivator}. The separability assumption on $M$ guarantees the existence of a section of the counit by Lemma~\ref{lem: section}. We know that each $M_J\Mod_{\mathbb D(J)}$ is additive and idempotent-complete since each $\mathbb D(J)$ is and $M_J$ is additive. Thus the Eilenberg-Moore adjunction satisfies all the conditions in the previous proposition and, hence, $\mathbb D$ is strong. For the last part, note that $F_M$ is cocontinuous as with all left adjoint morphisms of derivators (see~\cite[Proposition 2.9]{Groth}). Hence, in particular, $F_M$ is right exact. Dually, $U_M$ is left exact. Since these are morphisms between stable, strong derivators, by~\cite[Corollary 4.17]{Groth} it follows they are exact.
\end{proof}

\begin{rem}
	If the derivator $\mathbb D$ is big (see Definition~\ref{defn: big}), then the assumption of idempotent completeness is obsolete (see also Remark~\ref{rem:big implies idempotent complete}).
\end{rem}

\begin{rem}
	Assume $\mathbb D$ is a stable strong derivator of domain $\mathbf{Cat}$, and $M$ is a separable exact monad on $\mathbb D$. Then the last Corollary together with~\cite[Theorem 7.1]{linearity} implies that, when restricted to finite posets, the derivator~$M\Mod_{\mathbb D}$ is stable and strong, and that $F_M,U_M$ (see Proposition~\ref{prop: prederivator}) are exact morphisms of derivators. This reproves~\cite[Corollary 4.3]{Balmer} in the case the triangulated category~$\C$ is the base of a stable strong derivator $\mathbb D$ of domain $\mathbf{Cat}$ and~$M$ extends to a separable exact monad on $\mathbb D$.
\end{rem}

\begin{exam}	
	Assume $\mathbb D$ be a stable strong monoidal derivator. Then its underlying category $\C=\mathbb D(e)$ is a tensor triangulated category (see~\cite{GPS}). Write $\boxtimes$ for the external tensor product associated to the monoidal structure on $\mathbb D$, and let $A$ be a monoid object in $\C$. It is an easy but tedious verification to see that $M=A \boxtimes -: \mathbb D \to \mathbb D$ is a monad on $\mathbb D$, with multiplication and unit those of $A$. This monad is cocontinuous because the tensor product is assumed to be cocontinuous in each variable separately. Thus, by Corollary~\ref{cor: main}, we get a derivator $M\Mod_{\mathbb D}$ whose underlying category is exactly $A\Mod_{\C}$. If~$A$ is separable then so is $M$ and $M\Mod_{\mathbb D}$ is again a stable strong derivator. In particular $A\Mod_{\C}$ is a triangulated category. For examples of separable monoids in tensor triangulated categories we refer to the introduction.
\end{exam}

\appendix
\section{Some basics on derivators} \label{section: basics}
In this section, we gather some basic definitions that are needed throughout this paper. There are many sources for this, but our notation is the same as in~\cite{Groth}, to which the reader is referred for more details.

Let us recall first, that given a functor $u:J \to K$ and any $k \in K$, there is a category $J_{/k}$ where:
\begin{itemize}
	\item objects are pairs $(j,a)$, where $j$ an object of $J$ and $a:u(j) \to k$ is a morphism in $K$.
	\item morphisms $(j,a) \to (j',a')$ are those morphisms $f:j \to j'$ in $J$ such that $a=a' \circ u(f)$
\end{itemize}
There is also a projection functor $J_{/k} \to J$ mapping $(j,a)$ to $j$. The category $J_{k/}$ is defined dually and also comes equipped with a canonical projection to $J$. If $j \in J$ then these constructions applied to the identity functor on $J$ gives us categories $J_{/j}$ and $J_{j/}$ respectively. Finally, the fiber of $u$ over an object $k \in K$ is the category $J_k$ with objects those objects $j \in J$ such that $u(j)=k$ and morphisms $j \to j'$ those morphisms of $J$ mapping to the identity on $k$.

In the following we will denote by $\mathbf{Cat}$ the $2$-category of small categories while $\mathbf{CAT}$ refers to the $2$-"category" of all categories.

\begin{defn} \label{defn: diagram category}
	A full $2$-subcategory $\mathbf{Dia}$ of $\mathbf{Cat}$ is a \textbf{diagram category} if:
	\begin{enumerate}
		\item $\mathbf{Dia}$ contains all finite posets.
		\item $\mathbf{Dia}$ is closed under pullbacks and finite coproducts.
		\item For any $J \in \mathbf{Dia}$ and any $j \in J$ the slice constructions $J_{/j}$ and $J_{j/}$ also belong to $\mathbf{Dia}$.
		\item If $J \in \mathbf{Dia}$ then also $J^{\text{op}} \in \mathbf{Dia}$.
		\item If $u:J \to K$ is a Grothendieck fibration (see~\cite[Chapter 8]{Borceux2}) such that all the fibers $J_k \in \mathbf{Dia}$ and $K \in \mathbf{Dia}$, then also $J \in \mathbf{Dia}$.
	\end{enumerate}
\end{defn}

\begin{exam}
	Examples of diagram categories include $\mathbf{Cat}$, the $2$-category $\mathbf{Pos}_{\text f}$ of finite posets and $\mathbf{Dir}_{\text f}$ of finite direct categories.
\end{exam}

Given a $2$-category $K$, we will denote by $K^{\text{op}}$ the $2$-category obtained from $K$ by reversing $1$-cells.

\begin{defn}
	Let $\mathbf{Dia}$ a diagram category. A prederivator of domain $\mathbf{Dia}$ is a (strict) 2-functor~$\mathbb D: \mathbf{Dia}^{\text{op}} \to \mathbf{CAT}$.
\end{defn}

For the rest of this appendix, $\mathbf{Dia}$ will refer to a fixed diagram category, and all prederivators are of domain $\mathbf{Dia}$ unless specified otherwise.

\begin{exam} \label{prederivator}
	\begin{enumerate}
		\item Given a category $\C$, the assignment $J \mapsto \C^J$ gives us a prederivator, called the \textbf{prederivator represented by }$\mathbf{\C}$.
		\item A $\textbf{localizer}$ is a pair $(\C, \W)$ consisting of a category $\C$ and a class of arrows $\W$ in $\C$ called weak equivalences. Given a localizer $(\C, \W)$ and a small category $J$, let us denote by $\W_J$ the class of morphisms in $\C^J$ that are componentwise in $\W$. Assuming that the localization $\C^J[\W_J^{-1}]$ exists for each small category $J$, then the assignment $J \mapsto \C^J[\W_J^{-1}]$ defines a prederivator. For instance, this is true as long as our localizer is associated to a model category (see example~\ref{ex: derivator}).
	\end{enumerate}
\end{exam}

\begin{rem} \label{rem: coherent}
	Let $\mathbb D$ be a prederivator. We call $\mathbb D(e)$ the \textbf{underlying category of }$\mathbf{\mathbb D}$. Given a functor $u:J \to K$ between small categories, we will write $u^*$ for the image of $u$ under $\mathbb D$. The same notation applies to a natural transformation $\a:u\to v$.
	
	Given any category $J \in \mathbf{Dia}$ and any object $j \in J$, let us denote by $j$ again the functor $e \to J$ taking the unique object of $e$ to $j\in J$. We get a corresponding \textbf{evaluation functor} $j^*: \mathbb D(J) \to \mathbb D(e)$. Putting all evaluation functors together, we obtain a functor $\dia_J: \mathbb D(J) \to \mathbb D(e)^J$ which we call the \textbf{underlying diagram functor}. We think of $\mathbb D(J)$ as \textbf{coherent }$\mathbf{J}$\textbf{-shaped diagrams in }$\mathbf{\mathbb D}$, and of $\mathbb D(e)^J$ as \textbf{incoherent }$\mathbf{J}$\textbf{-shaped diagrams in }$\mathbf{\mathbb D}$.
	
	Similarly, given two small categories $J,K$, we can define a \textbf{partial underlying diagram functor} $\dia_{J,K}: \mathbb D(K \times J) \to \mathbb D(J)^K$. We think of $\mathbb D(J)^K$ as $J \times K$-shaped diagrams in $\mathbb D$ that are coherent in the $J$-direction, and incoherent in the $K$-direction. See also~\cite[p. 323]{Groth}.
\end{rem}

\begin{defn}
	A prederivator $\mathbb D$ \textbf{admits (homotopy) right Kan extensions} if for any functor $u:J \to K$ in $\mathbf{Dia}$, the induced functor $u^*: \mathbb D(K) \to \mathbb D(J)$ has a right adjoint $u_*$. Dually, if for any such $u$ the functor $u^*$ has a left adjoint $u_!$, then we say $\mathbb D$ \textbf{admits (homotopy) left Kan extensions}.
\end{defn}

Now let us consider a square in $\mathbf{Dia}$:
\begin{equation} \label{eq: square}
\begin{tikzcd}
J \ar[r, "f"] \ar[d, "u"'] & J' \dar["u'"] \dlar[phantom, sloped, "\xLeftarrow{\a}"] \\
K \rar["g"'] & K'
\end{tikzcd}
\end{equation}
\ie we have a natural transformation $\a:u'f \Rightarrow gu$. Given any prederivator $\mathbb D$, applying it to this square gives a new one as follows:
\begin{equation*}
\begin{tikzcd}
\mathbb D(J)  & \mathbb D(J') \lar["f^*"'] \dlar[phantom, sloped, "\xLeftarrow{\a^*}"] \\
\mathbb D(K) \uar["u^*"] & \mathbb D(K') \lar["g^*"] \uar["u'^*"']
\end{tikzcd}
\end{equation*}

Assuming that $\mathbb D$ admits left Kan extensions, the functors $u^*,u'^*$ have left adjoints. The \textbf{Beck-Chevalley transform} (or \textbf{mate}) $\a_!:u_!f^* \to g^*u_!'$ of $\a^*$ is the natural transformation given by the pasting:
\begin{equation*}
\begin{tikzcd}
\mathbb D(K) &
\mathbb D(J) \lar["u_!"']  &
\mathbb D(J') \lar["f^*"']  \ar[ld, phantom, sloped, "\xLeftarrow{\a^*}"] & \\
&
\mathbb D(K) \ar[lu, bend left, ""'{name=A}, "1_{\mathbb D(K)}"]  \uar["u^*"']  &
\mathbb D(K') \uar["u'^*"] \lar["g^*"] &
\mathbb D(J') \lar["u'_!"] \ar[ul, bend right, "1_{\mathbb D(J')}"', ""{name=B}]
\ar[phantom, sloped, from=A, to=1-2, "\xLeftarrow{\e}"] \ar[phantom, sloped, from=2-3, to=B, "\xLeftarrow{\eta}"]
\end{tikzcd}
\end{equation*}
\ie it is given as the top arrow in the following defining commutative square:
\begin{equation*}
\begin{tikzcd}
u_!f^* \dar["u_!f^* \eta"'] \rar["\a_!"] &[+25pt] g^*u'_! \\
u_!f^*u'^*u'_! \rar["u_! \a^* u'_!"'] & u_!u^*g^*u'_! \uar["\e g^* u'_!"']
\end{tikzcd}
\end{equation*}

Dually, if $\mathbb D$ admits homotopy right Kan extensions, then $\a^*$ has a mate $\a_*: u'^*g_* \to f_*u^*$, given by a similar pasting as above. We refer the reader to~\cite[\S 1.2]{Groth} for more details and properties of these constructions. As a last preparation, we give the following definition:

\begin{defn}
	Let $\mathbb D$ be a prederivator that admits homotopy left (respectively right) Kan extensions. A square~\eqref{eq: square} in $\mathbf{Dia}$ is called $\mathbb D$\textbf{-exact} if the mate $\a_!$ (respectively $\a_*$) is an isomorphism.
\end{defn}

Note that by~\cite[Lemma 1.14]{Groth}, if $\mathbb D$ admits both left and right homotopy Kan extensions, then the two (seemingly different) definitions above agree.

\begin{defn} \label{defn: left derivator}
	A prederivator $\mathbb D$ is a \textbf{left derivator} if it satisfies the following axioms:
	
	\noindent
	($Der1$) For any finite family of categories $\{J_i\}$ in $\mathbf{Dia}$, the induced functor $\mathbb D(\coprod J_i) \to \prod \mathbb D(J_i)$ is an equivalence of categories.
	
	\noindent
	($Der2$) For any category $J \in \mathbf{Dia}$ the underlying diagram functor $\dia_J:\mathbb D(J) \to \mathbb D(e)^J$ is conservative. That is, a morphism $f:X \to Y$ in $\mathbb D(J)$ is an isomorphism if and only if $j^*f$ is an isomorphism in $\mathbb D(e)$ for all $j \in J$.
	
	\noindent
	($Der3L$) $\mathbb D$ admits right Kan extensions.
	
	\noindent
	($Der4L$) Given a functor $u:J \to K$ in $\mathbf{Dia}$ and any $k \in K$, the following square is $\mathbb D$-exact (meaning $\a_*:k^*u_* \to (p_{J_{k/}})_*pr^*$ is an isomorphism): 
	\begin{equation} \label{eq: pointwise formula right Kan extension}
	\begin{tikzcd}
	J_{k/} \ar[r, "pr"] \ar[d, "p_{J_{k/}}"'] & J \ar[d, "u"] \dlar[phantom, sloped, "\xRightarrow{\a}"] \\
	e \ar[r, "k"'] & K
	\end{tikzcd}
	\end{equation}
	where $pr:J_{k/} \to J$ is the canonical projection as in the beginning of this appendix, and the natural transformation $\a$ is the obvious one whose component at a pair $(j,a)$ is exactly $a$.
\end{defn}

\begin{defn}
	A prederivator $\mathbb D$ is a \textbf{right derivator} if it satisfies $Der1$ and $Der2$ above and in addition the following axioms:
	
	\noindent
	($Der3R$) $\mathbb D$ admits left Kan extensions.
	
	\noindent
	($Der4R$) Given any functor $u:J \to K$ in $\mathbf{Dia}$, and any $k \in K$ the following square is $\mathbb D$-exact (meaning $\a_!$ is an isomorphism):
	\begin{equation} \label{eq: pointwise formula left Kan extension}
	\begin{tikzcd}
	J_{/k} \ar[r, "pr"] \ar[d, "p_{J_{/k}}"'] & J \ar[d, "u"] \dlar[phantom, sloped, "\xLeftarrow{\a}"] \\
	e \ar[r, "k"'] & K
	\end{tikzcd}
	\end{equation}
	where $pr$ is the canonical projection and the natural transformation is again the obvious one whose component at $(j,a)$ is exactly $a$.
\end{defn}

\begin{defn}
	A \textbf{derivator} is a prederivator $\mathbb D$ that is both a left and a right derivator.
\end{defn}

\begin{rem}
	Some authors require an infinite version of axiom $Der1$. The following definition is taken from~\cite{Fritz}:
\end{rem}
\begin{defn} \label{defn: big}
	A derivator of domain $\mathbf{Dia}$ is \textbf{big} if $\mathbf{Dia}$ is closed under infinite coproducts and $Der1$ holds for arbitrary families (instead of just finite ones).
\end{defn}

\begin{rem}
	Let us recall that in ordinary category theory, a left Kan extension of a functor $X:J \to \C$ along a functor $u:J \to K$ is a pair $(u_!(X):K \to \C, \eta:X \Rightarrow u_!(X) \circ u)$ that is initial among such pairs: if $(L,\theta)$ another such pair, there is a unique natural transformation $\delta:u_!(X) \to L$ such that $\delta u \circ_v \eta=\theta$ (here $\circ_v$ refers to vertical composition). In particular, when $K=e$, this is nothing more than the colimit of $X$. Thus, ordinary left Kan extensions are generalizations of colimits. It is a known fact that	if the category $\C$ is cocomplete, then for all functors $u: J \to K$ and all diagrams $X: J \to \C$ the left Kan extension of $X$ along $u$ exists; in fact the assignment $X \mapsto u_!(X)$ defines a functor $u_!:\C^J \to \C^K$ which is left adjoint to $u^*:=- \circ u:\C^K \to \C^J$; furthermore for any object $k \in K$ we have a canonical isomorphism $u_!(X)_b \cong \text{colim}_{J/k} F \circ pr$ where $pr:J/k \to J$ the canonical functor. Note also that in a cocomplete category, all $I$-shaped colimits (for $I$ a small category) assemble to a functor $\text{colim}_I= \pi_{I,!}:\C^I \to \C$ where $\pi_I:I \to e$ is the unique functor.
	
	Given a prederivator $\mathbb D$, we care more about coherent diagrams instead of incoherent ones (see Remark~\ref{rem: coherent}). Therefore, existence of left Kan extensions can be expressed as existence of certain adjoints to pullback functors between \textit{coherent} diagrams, analogously to the classical case described above. Then $Der3R$ says exactly that the homotopy theory expressed by $\mathbb D$ has all homotopy left Kan extensions; in particular all homotopy colimits. On the other hand, $Der4R$ says that homotopy left Kan extensions are pointwise given as homotopy colimits over certain slice categories (the homotopy analogue of the pointwise formulas above).
\end{rem}

\begin{exam} \label{ex: derivator}
	\begin{enumerate}
		\item If $\C$ is a category, then the prederivator represented by $\C$ (see Example~\ref{prederivator}) is a derivator if and only if $\C$ is bicomplete.
		\item Given a combinatorial model category $M$ with $\W$ as weak equivalences, the prederivator associated to the localizer $(M,\W)$ (see Example~\ref{prederivator}) is actually a derivator (see for instance~\cite[Proposition 1.30]{Groth}). More generally this is true for any model category (see~\cite{Cisinski-Model}).
		\item Given an exact category $\mathcal E$, the assignment $J \mapsto D^b(\mathcal E^J)$ defines a derivator of domain $\mathbf{Dir_f}$ (see \cite{Keller-exact}). Here $D^b$ stands for the bounded derived category of an exact category.
	\end{enumerate}
\end{exam}

\section{2-categorical notions in derivators} \label{section: 2-stuff}
In~\cite[Chapter 2]{Groth} it is shown that there is a $2$-category of prederivators $\PDer$ with morphisms pseudonatural transformations, and $2$-morphisms modifications. For the reader's convenience we spell out what this means:

Given two prederivators $\mathbb D, \mathbb E$ of the same domain $\mathbf{Dia}$, a morphism $\mathbb D \to \mathbb E$ consists of the following data:
\begin{itemize}
	\item For any category $J \in \mathbf{Dia}$, a functor $F_J: \mathbb D(J) \to \mathbb E(J)$.
	\item For any functor $u:J \to K$ in $\mathbf{Dia}$, a natural isomorphism $\gamma^F_u:u^*F_K \xrightarrow{\cong} F_J u^*$ as indicated in the diagram:
	\begin{equation*}
	\begin{tikzcd}
	\mathbb E(J) & \mathbb E(K) \lar["u^*"'] \dlar[phantom, sloped, "\xLeftarrow{\gamma^F_u}"] \\
	\mathbb D(J) \uar["F_J"] & \mathbb D(K) \uar["F_K"'] \lar["u^*"]
	\end{tikzcd}
	\end{equation*}
\end{itemize}
subject to the following coherence conditions:
	\begin{enumerate}
		\item For any category $J \in \mathbf{Dia}$ we have $\gamma^F_{1_J}=1_{F_J}$.
		\item Given a pair of composable functors $J \xrightarrow{u} K \xrightarrow{v}L $ in $\mathbf{Dia}$, the following pastings are equal: 
		\begin{equation*}
		\begin{tikzcd}
		\mathbb E(J) & \mathbb E(K) \dlar[phantom, sloped, "\xLeftarrow{\gamma^F_u}"] \lar["u^*"'] & \mathbb E(L) \lar["v^*"'] \dlar[phantom, sloped, "\xLeftarrow{\gamma^F_v}"] & \mathbb E(J) \dlar[phantom,"="] & \mathbb E(L) \lar["(vu)^*"'] \dlar[phantom, sloped, "\xLeftarrow{\gamma^F_{vu}}"] \\
		\mathbb D(J) \uar["F_J"] & \mathbb D(K) \uar["F_K"'] \lar["u^*"]  & \mathbb D(L) \uar["F_L"'] \lar["v^*"] & \mathbb D(J) \uar["F_J"] & \mathbb D(L) \lar["(vu)^*"] \uar["F_L"']
		\end{tikzcd}
		\end{equation*}
		
		In other words we have a commutative diagram:
		\begin{equation} \label{eq: 1-coherence}
		\begin{tikzcd}
		u^*v^*F_L \drar["\gamma^F_{vu}"'] \rar["u^* \gamma^F_v"] & u^*F_Kv^* \dar["\gamma^F_u v^*"] \\
		&F_Ju^*v^*
		\end{tikzcd}
		\end{equation}
		\item Given a natural transformation
		$\begin{tikzcd} J \rar[bend left=40, "u", ""'{name=U}] \rar[bend right=40, "v"', ""{name=V} ] & K \arrow[Rightarrow, from=U, to=V, "\alpha"]\end{tikzcd}$ 
		in $\mathbf{Dia}$, the following pastings are equal:
		\begin{equation*}
		\begin{tikzcd}[row sep=large] 
		\mathbb E(J) & \mathbb E(K) \lar["u^*"']  \dlar[phantom, sloped, "\xLeftarrow{\gamma^F_{u}}"] & \mathbb E(J) \dlar[phantom,"="] & \mathbb E(K) \lar[bend left=20,"v^*"'{name=V}] \lar[bend right=50, ""{name=U},"u^*"'] \dlar[phantom, sloped, "\xLeftarrow{\gamma^F_{v}}"] \arrow[Rightarrow, from=U, to=V, "\alpha^*"] \\
		\mathbb D(J) \uar["F_J"] & \mathbb D(K) \lar[bend right=20,"u^*"{name=U}] \lar[bend left=50,"v^*",""'{name=V}] \uar["F_K"'] \arrow[Rightarrow, from=U, to=V,"\alpha^*"] & \mathbb D(J) \uar["F_J"] & \mathbb D(K)  \uar["F_K"'] \lar["v^*"''] 
		\end{tikzcd}
		\end{equation*}
In other words we have a commutative diagram:
		\begin{equation} \label{eq: 2-coherence}
		\begin{tikzcd}
		u^*F_K \rar["\alpha^*F_K"] \dar["\gamma^F_u"']& v^*M_K \dar["\gamma^F_v"] \\
		F_Ju^* \rar["F_J\alpha^*"'] & F_Jv^*
		\end{tikzcd}
		\end{equation}
	\end{enumerate}
	Such a morphism is called $\mathbf{strict}$ if all its coherence isomorphisms are actually identities.
	
	Given $2$ morphisms $F,G$ from a prederivator $\mathbb D$ to a prederivator $\mathbb E$ (both of the same domain $\mathbf{Dia}$), a $2$-cell (or modification) $\rho$ from $F$ to $G$ consists of a natural transformation $\rho_J:F_J \Rightarrow G_J$ for each category $J \in \mathbf{Dia}$ compatible with the coherence isomorphisms of $F,G$. This means more explicitly that for any functor $u:J \to K$ in $\mathbf{Dia}$, the following pastings are equal:
			\begin{equation*}
			\begin{tikzcd}[row sep=large] 
			\mathbb E(J) & \mathbb E(K) \lar["u^*"']  \dlar[phantom, sloped, "\xLeftarrow{\gamma^F_{u}}"] & \mathbb E(J) \dlar[phantom,"="] & \mathbb E(K) \lar["v^*"'] \dlar[phantom, sloped, "\xLeftarrow{\gamma^G_{u}}"] \\
			\mathbb D(J) \uar["F_J"',""{name=V}] \uar[bend left=70, "G_J",""{name=U}] \arrow[Leftarrow, from=U, to=V,"\rho_J"] & \mathbb D(K) \lar["u^*"] \uar["F_K"']  & \mathbb D(J) \uar["G_J"] & \mathbb D(K)  \uar["G_K",""'{name=U}] \uar[bend right=70, "F_K"',""{name=V}] \lar["u^*"''] \arrow[Leftarrow, from=U, to=V,"\rho_K"]
			\end{tikzcd}
			\end{equation*}
	\begin{equation} \label{eq: modification}
	\begin{tikzcd}
	u^*F_K \rar["u^* \rho_K"] \dar["\gamma^{F}_{u}"'] & u^*G_K \dar["\gamma^G_{u}"''] \\
	F_Ju^* \rar["\rho_J u^*"'] & G_Ju^*
	\end{tikzcd}
	\end{equation}

\bibliographystyle{alpha}
\bibliography{Monads-Final-Version}

\end{document}